\documentclass[12pt]{elsarticle}
\usepackage{amsmath, amssymb, amsthm}
\usepackage{graphicx,amsfonts,amsbsy}
\usepackage{mathtools,accents}
\usepackage{hyperref}
\usepackage[utf8]{inputenc}
\usepackage{floatrow}
\usepackage{float}
\usepackage{subfig}
\usepackage{amsthm}
\usepackage{epstopdf}
\usepackage{anysize}
\usepackage{color,soul}
\usepackage{multirow}
\newtheorem{theorem}{Theorem}[section]
\newtheorem{lemma}[theorem]{Lemma}

\newtheorem{definition}[theorem]{Definition}

\DeclareMathOperator{\e}{e}

\journal{ }
\input epsf.sty
\usepackage{lineno}
\begin{document}

\begin{frontmatter}
\title{Discontinuous harvesting policy in a Filippov system involving prey refuge}

\author[label1]{Rajesh Ranjan Patra\corref{cor1}}\ead{rajeshp911@yahoo.com}
\author[label1]{Sarit Maitra}

\address[label1]{Department of Mathematics, NIT Durgapur, Durgapur-713209, India}
\cortext[cor1]{Corresponding author}

\begin{abstract}
In this article, a non-smooth predator-prey dynamical system is considered. Here, we discuss about sustainable harvesting in a Filippov predator-prey system, which can produce yield and at the same time prevent over-exploitation of bioresources. The local and global stability analysis of the two subsystems, with and without harvesting, are studied. Furthermore, for the Filippov system, we have performed bifurcation analysis for several key parameters like predation rate, threshold quantity and prey refuge.  Some local sliding bifurcations are also observed for the system. Numerical simulations are presented to illustrate the dynamical behaviour of the system.

\end{abstract}
\end{frontmatter}

\section{Introduction}
Non-smoothness is an integral part of mathematical modelling related to different physical systems and engineering applications, with non-smooth dynamical systems finding applications across diverse fields \cite{giannakopoulos2002, bernardo2008, biak2013}. These systems exhibit a diverse range of complex phenomena that go beyond the scope of classical theory. Filippov systems are a sub-class of non-smooth systems which consist of different smooth systems defined for various open regions and the regions are separated by smooth boundaries \cite{bernardo2008,filippov1988}. These are valuable in ecological modelling as they effectively capture essential ecological phenomena such as threshold effects, abrupt transitions, and discontinuous behavior in species interactions. 

Harvesting of species has always been a topic of significant importance due to its ecological and economic perspectives \cite{clark1985,khamis2011}. There are various studies on constant harvesting \cite{wang2006,luo2017}, continuous linear harvesting \cite{zhou2013,meng2021} and many non-linear harvesting methods \cite{gupta2013,dubey2014,majumdar2022} that use continuous functions to measure the proportion of population to be exploited. But many authors have confirmed in the past that, over-exploitation can result in population extinction \cite{lv2013,zhao2015}. For example, continuous harvesting of herbivorous fish together with sustained predation pressure by lionfish leads to a drop in herbivorous fish population \cite{lv2013}, 
resulting in a significant ecological imbalance in coral reefs. Hence, researchers are actively focusing on the challenges related to the sustainable harvesting of renewable bio-resources, which can help yield resources and limit the level of exploitation. One way to achieve this is by implementing a discontinuous harvesting policy which can be represented by a Filippov system.

In ecology, Filippov systems are often used for models involving epidemic models \cite{wang2014,wang2018}, integrated pest management \cite{tan2016,zhou2022}, and discontinuous harvesting \cite{
bondarev2022,luo2023}, etc. The threshold parameter plays a crucial role in determining when harvesting should occur based on the density of one or more populations. Such a level can be set to maximize the total yield from the resource or minimize the impact of harvesting on the ecosystem.

The investigation of how the hiding behavior of prey influences the dynamics of predator-prey interactions is widely recognized as a significant topic in applied mathematics and theoretical ecology \cite{holling1959,hassell1974,hassell1978}. Previous research on the topic indicates that prey refuges play a critical role in stabilizing predator-prey interactions \cite{ko2006,ji2010}. In the studies \cite{kar2005} and \cite{chen2010},  the authors showed that refuges can prevent prey extinction and influence population outbreaks. Another observation by Mondal and Samanta \cite{mondal2020} shows that while the refuge strengthens the prey density, it may also lead to the extinction of the predator species, even when the predator has access to supplementary food resources. Therefore, prey refuge is an essential aspect in the survival and persistence of populations.

The article is organized as follows. Section \ref{secmodel} shows the formulation of the population model while section \ref{basic_results} contains the basic results like boundedness of the solutions and existence of positive equilibria of the system. Section \ref{secsubsys1} and \ref{secsubsys2} includes the stability analysis of non-harvesting subsystem and harvesting subsystem respectively. The Filippov system is analyzed in detail in section \ref{secfilippov}. Numerical simulations involving results on sliding bifurcation are presented in section \ref{secnum}. Finally, the summary of the work is discussed in section \ref{seccon}.

\section{The Model}\label{secmodel}
In this section, we layout a mathematical model involving a predator and a prey population, with both populations are subjected to harvesting. It is assumed that both populations live in a closed habitat with limited reources and hence their intrinsic growth terms are modelled by logistic growth functions. The predator is assumed to be a generalist, so it has access to alternate food sources present in the environment. The predation term is a Holling type II function which includes the handling time for the predators. Also, it is considered that an $m-$fraction of the prey population is spared from predation. This forms prey refuge. Then instead of the total prey population, only $(1-m)-$fraction of the prey population is exposed to predation. The dynamics of the generalist species is modelled following the representation by Guin et al. \cite{guin2012} and Sen et al. \cite{sen2020}. Let $x(t)$ and $y(t)$ represent the density of prey and predator population at any instant of time $t$, then combining all these factors, the proposed mathematical model is given by
\begin{eqnarray}\label{sys}
\frac{dx}{dt} &=& r_1x\left(1-\frac{x}{k_1}\right)-\frac{p(1-m)xy}{1+b(1-m)x}-\psi q_1Ex,\nonumber\\
\frac{dy}{dt} &=& r_2y\left(1-\frac{y}{k_2}\right)+\frac{ep(1-m)xy}{1+b(1-m)x}-\psi q_2Ey,
\end{eqnarray}
where the parameters used in \eqref{sys} are given in Table \ref{paramter_table}. The value of the dimensionless quantity $\psi$ switches according to the abundance of the prey species and is given by
\begin{equation}\label{psi}
\psi = \begin{cases}
0, & x <S,\\
1, & x >S.
\end{cases}
\end{equation}
Here, $S$ is the critical prey population density level that regulates the harvesting of the populations. Certainly, harvesting continues for $x>S$ given that $\psi=1$. Similarly, harvesting does not occur when $x<S$ since $\psi=0$. This ensures that harvesting is prohibited when the prey population density falls below the critical value $S$. The predator has multiple food sources, with the prey being one of them. By imposing the condition of harvesting based on prey density, as defined in \eqref{psi}, the prey population is encouraged to maintain a healthy level, which can indirectly ensure maintaining at least some population of predators by ensuring food for them at a certain level. So we can assume that by maintaining a threshold for prey population, we can indirectly maintain a predator population threshold. 
Thus, harvesting is presumed to depend solely on prey density.\\
System \eqref{sys} can be rewritten as
\begin{table}
\centering
\caption{Parameters and their biological meaning}
\begin{tabular}{||c|c|c||}
\hline\hline
Parameter & Biological meaning & Unit\\
\hline
\hline
$r_1$ & intrinsic growth rate of the prey population & [time]$^{-1}$\\
$r_2$ & intrinsic growth rate of the predator population & [time]$^{-1}$\\
$k_1$ & the carrying capacity for the prey & [density]\\
$k_2$ & the carrying capacity for the predator & [density]\\
$p$ & the maximum predation rate & [density]$^{-1}$[time]$^{-1}$\\
$m$ & represents the fraction of prey that form the refuge; $m\in (0,1)$ & dimensionless\\
$b$ & involves the handling time of the predators & [density]$^{-1}$\\
$q_1$ & the catchability coefficient for the prey & [time]$^{-2}$)\\
$q_2$ & the catchability coefficient for the predator & [time]$^{-2}$)\\
$E$ & the harvesting effort & [time]\\
$e$ & converison rate of prey population to biomass & dimensionless\\
\hline\hline
\end{tabular}
\label{paramter_table}
\end{table}
\begin{equation}
\dot{z}(t) = \begin{cases}
\frac{dz(t)}{dt}=F_{S_1}(z), & z\in S_1,\\
\frac{dz(t)}{dt}=F_{S_2}(z), & z\in S_2,
\end{cases}
\end{equation}
for $z(t)=(x(t),y(t))^T$, where
\begin{eqnarray*}
F_{S_1}(z) &=& \left(r_1x\left(1-\frac{x}{k_1}\right)-\frac{p(1-m)xy}{1+b(1-m)x}, r_2y\left(1-\frac{y}{k_2}\right)+\frac{ep(1-m)xy}{1+b(1-m)x} \right)^T,\\
F_{S_2}(z) &=& \left(r_1x\left(1-\frac{x}{k_1}\right)-\frac{p(1-m)xy}{1+b(1-m)x}-q_1Ex, r_2y\left(1-\frac{y}{k_2}\right)+\frac{ep(1-m)xy}{1+b(1-m)x}-q_2Ey \right)^T,
\end{eqnarray*}
and 
\begin{equation}\label{s1s2}
S_1 = \left\{z\in\mathbb{R}_2^+: P(z)<0\right\},\quad \& \quad S_2=\left\{z\in\mathbb{R}_2^+: P(z)>0\right\},
\end{equation}
where
\begin{equation}\label{pz}
P(z) = \alpha^Tz-S,
\end{equation}
for $\alpha=(1,0)^T$.
Let $\Sigma$ be the $(n-1)-$dimensional switching manifold separating the two regions $S_1$ \& $S_2$, where $n$ denotes the number of state variables in the system, then
\begin{equation*}
\Sigma = \left\{z\in \mathbb{R}_2^+: P(z)=0\right\}.
\end{equation*}
Following Filippov \cite{filippov1988}, the solution of system \eqref{sys} can be defined as the solutions of the differential inclusion
\begin{equation*}
\frac{dz(t)}{dt}\in F(z)=F_{S_1}+\psi (F_{S_2}-F_{S_1}),
\end{equation*}
where $\psi=0$ when $P(z)<0$, $\psi=1$ when $P(z)>0$, and $\psi\in (0,1)$ when $P(z)=0$.
\section{Basic results: boundedness, existence \& stability equilibria, and subsystems}\label{basic_results}
In this section, we study the boundedness of the solutions of system \eqref{sys}, and the existence of interior equilibria of the vector fields $F_{S_1}$ and $F_{S_2}$. The harvesting and non-harvesting subsystems of the Filippov system \eqref{sys} are also examined, and  the sufficient conditions for the existence and stability of interior equilibria are obtained. We state the following lemma that will help us to find the upper bound of the solutions of the given system.
\begin{lemma}[\cite{aziz2002}]\label{lemma_bound}
Let $\eta(t)$ be an absolute continuous function with real numbers $k_1$ and $k_2$ with $k_1 \neq 0$ s.t. $d\eta/dt+k_1 \eta(t) \le k_2,~ \forall t\ge 0$, then for all $t\ge \tilde{T} \ge 0$, $$\eta(t) \le \frac{k_2}{k_1}-\left(\frac{k_2}{k_1}-\eta(\tilde{T}) \right) \e^{-k_1 (t-\tilde{T})}.$$
\end{lemma}
\subsection{Boundedness}
\noindent From the first equation of \eqref{sys}, we have
$x'\leq r_1x\left(1-x/k_1\right)$.
Hence, we deduce
\begin{equation*}
x(t)\leq k_1
\end{equation*}
for some $T_x\geq t\geq 0$, using 
Lemma \ref{lemma_bound}. Now, for the upper bound of $y(t)$, we consider different cases according the value of $\psi$. 
Let $\psi=0$. Then the second equation of \eqref{sys} becomes $y'\leq r_2y(1-y/k_2)+(ep(1-m)k_1y)/(1+b(1-m)k_1)$
which gives
\begin{equation*}
y(t)\leq \frac{k_2}{r_2}\left(r_2+\frac{ep(1-m)k_1}{1+b(1-m)k_1}\right), 
\end{equation*}
for some $T_{y_1}\geq t\geq 0$. For $\psi=1$, let $\eta(t)$ be a continuous function such that $\eta(t)=ex(t)+y(t)$, then we have
$\eta'(t)\leq -\gamma\eta+(r_1k_1e+r_2k_2)/4$,
which along with Lemma \ref{lemma_bound} implies
\begin{equation*}
y(t)\leq \eta(t)\leq \frac{1}{4\gamma}(r_1k_1e+r_2k_2),
\end{equation*}
for some $T_{y_2}\geq t\geq 0$, where $\gamma=\min\{E_1e,E_2\}$. Hence, we have 
\begin{equation*}
y(t)\leq \max\left\{ \frac{k_2}{r_2}\left(r_2+\frac{ep(1-m)k_1}{1+b(1-m)k_1}\right),\frac{1}{4\gamma}(r_1k_1e+r_2k_2) \right\},
\end{equation*}
for $T_y\geq t\geq 0$, where $T_y=\max\{T_{y_1},T_{y_2}\}$.
\subsection{Existence of interior equilibrium point}
In this section, we investigate the equilibria of system \eqref{sys}, which can be classified into three types: trivial, semi-trivial, and interior equilibria. The phase plane of the system features two distinct vector fields, denoted by $F_{S_1}$ and $F_{S_2}$, that are separated by the switching manifold $x=S$. The equilibria that arise when $\psi=0$ correspond to those of $F_{S_1}$, while the equilibria that arise when $\psi=1$ correspond to those of $F_{S_2}$.

The system given by \eqref{sys} has a trivial equilibrium at the origin and may have up to two semi-trivial equilibria, which depend on the parameter $\psi$ (0 or 1) and require that the non-zero coordinates have a positive sign. These semi-trivial equilibria are located at $(k_1(r_1-\psi q_1E)/r_1,0)$ for $r_1>\psi q_1E$ and $(0,k_2(r_2-\psi q_2E)/r_2)$ for $r_2>\psi q_2E$.
\begin{theorem}\label{thm_eq_exist}
A unique positive interior equilibrium point exists in either of the systems $\dot{z}=F_{S_1}$ (for $\psi=0$) or $\dot{z}=F_{S_2}$ (for $\psi=1$), provided that the condition $r_1>\psi q_1E+r_1x^*/k_1$ is met by the respective system, along with any one of the following conditions:
$$\text{(i) }\alpha_0<0, \alpha_1<0\;\& \;\alpha_2<0;\quad\text{(ii) } \alpha_0<0, \alpha_1<0\;\& \;\alpha_2>0;\quad\text{(iii) } \alpha_0<0, \alpha_1>0\;\& \;\alpha_2>0;$$
where $\alpha_i$'s are given in the proof.
\end{theorem}
\begin{proof}
\noindent An interior equilibrium point of system \eqref{sys} will simultaneously satisfy
\begin{eqnarray}\label{sys_eq}
r_1\left(1-\frac{x}{k_1}\right)-\frac{p(1-m)y}{1+b(1-m)x}&=&\psi q_1E,\nonumber\\
r_2\left(1-\frac{y}{k_2}\right)+\frac{ep(1-m)x}{1+b(1-m)x}&=&\psi q_2E.
\end{eqnarray}
Eliminating $y$ from the first equation, we have
\begin{equation}\label{y_eq}
y = \frac{1}{p(1-m)}(r_1-\psi q_1E-r_1x/k_1)(1+b(1-m)x).
\end{equation}
Now, replacing the value of $y$ in the first equation of \eqref{sys_eq}, it reduces to
\begin{equation}\label{xroot}
\alpha_3x^3+\alpha_2x^2+\alpha_1x+\alpha_0=0,
\end{equation}
where
\begin{eqnarray*}
\alpha_0&=&k_1k_2p(1-m)(r_2-\psi q_2E)-k_1r_2(r_1-\psi q_1E),\\
\alpha_1&=&r_1r_2-2k_1r_2(r_1-\psi q_1E)b(1-m)+k_1k_2p(1-m)^2{[b(r_2-\psi q_2E)+ep]},\\
\alpha_2&=&2r_1r_2b(1-m)-k_1r_2b^2(1-m)^2(r_1-\psi q_1E),\\
\alpha_3&=&r_1r_2b^2(1-m)^2>0.
\end{eqnarray*}
Using Descarte's rule of sign in \eqref{xroot}, a unique positive root, $x^*$, exists if $\alpha_i$'s, for $i=0,1,2,3$, satisfy any of the following combinations:
$$\text{(i) }\alpha_0<0, \alpha_1<0\;\& \;\alpha_2<0;\quad\text{(ii) } \alpha_0<0, \alpha_1<0\;\& \;\alpha_2>0;\quad\text{(iii) } \alpha_0<0, \alpha_1>0\;\& \;\alpha_2>0.$$
Thus another positive quantity $y^*$ is obtained by putting $x=x^*$ in \eqref{y_eq}, provided $r_1>\psi q_1E+r_1x^*/k_1$. Therefore, $(x^*,y^*)$ becomes a unique interior equilibrium point.
\end{proof}
The dynamical analysis of the Filippov system in \eqref{sys} requires an investigation of its two subsystems. Consequently, we proceed to investigate the dynamics of the corresponding subsystems.
\subsection{Analysis of subsystem-I}\label{secsubsys1}
In this subsection, we perform a brief analysis of a subsystem of the Filippov system \eqref{sys} when no harvesting policy is imposed on the predator and the prey. The subsystem is given  by the differential equation
\begin{equation}\label{subsys1}
\dot{z}(t)=F_{S_1}(z),
\end{equation}
which is obtained by putting $\psi=0$ in system \eqref{sys}. Subsystem \eqref{subsys1} features four equilibria: one trivial equilibrium at $(0,0)$, two axial equilibria at $(k_1,0)$ and $(0,k_2)$, and one interior equilibrium at $(x^*,y^*)$. Our focus is on the equilibrium where both species coexist, namely $(x^*,y^*)$. The existence of this equilibrium is demonstrated in Theorem \ref{thm_eq_exist} (for $\psi=0$). Subsequently, we explore its local and global stability properties. 

Using the linearization method for local stability analysis, the Jacobian of subsystem \eqref{subsys1} evaluated at the interior equilibrium point $(x^*,y^*)$ is given by
\begin{equation*}
J(x^*,y^*)=\begin{pmatrix}
-x^*\left[\dfrac{r_1}{k_1}-\dfrac{bp(1-m)^2y^*}{(1+b(1-m)x^*)^2}\right] & -\dfrac{p(1-m)x^*}{1+b(1-m)x^*}\\
\dfrac{ep(1-m)x^*}{(1+b(1-m)x^*)^2} & -\dfrac{r_2}{k_2}y^*
\end{pmatrix}.
\end{equation*}
\begin{theorem}\label{thm_local}
The interior equilibrium of the subsystem in \eqref{subsys1} is locally asymptotically stable if $r_1/k_1>{bp(1-m)^2y^*}/{(1+b(1-m)x^*)^2}$ holds true.
\end{theorem}
\begin{proof}
If the given condition is satisfied, the Jacobian matrix will have a negative trace and a positive determinant. This guarantees that the corresponding characteristic polynomial will have two negative roots. Hence, the equilibrium point will be locally asymptotically stable.
\end{proof}
\begin{theorem}\label{thm_global}
The interior equilibrium of the subsystem in \eqref{subsys1} is globally asymptotically stable if ${r_1}/{k_1}>{bp(1-m)^2y^*}/{(1+b(1-m)x^*)}$ holds true.
\end{theorem}
\begin{proof}
Let us consider the positive semi-definite function
\begin{equation}\label{lyap_fun}
V(x(t),y(t))=e\left(x-x^*-x^*\; \ln{\frac{x}{x^*}}\right)+(1+b(1-m)x^*)\left(y-y^*-y^*\; \ln{\frac{y}{y^*}}\right)
\end{equation}
which is equivalent to zero only at the equilibrium $(x^*,y^*)$ and positive at any other value of $x$ \& $y$. Now, differentiating this with respect to time $t$, we have
\begin{equation*}
V'=e(x-x^*)\frac{x'}{x}+(1+b(1-m)x^*)(y-y^*)\frac{y'}{y}.
\end{equation*}
Evaluating $V'$ along the solutions of the system, we get
\begin{equation*}
V'=-e\left[\frac{r_1}{k_1}-\frac{bp(1-m)^2y^*}{[1+b(1-m)x^*][1+b(1-m)x]}\right](x-x^*)^2-\frac{r_2}{k_2}(1+b(1-m)x^*)(y-y^*)^2.
\end{equation*}
As we have $x\geq 0$, so the above expression reduces to
\begin{equation*}
V'\leq -e\left[\frac{r_1}{k_1}-\frac{bp(1-m)^2y^*}{1+b(1-m)x^*}\right](x-x^*)^2-\frac{r_2}{k_2}(1+b(1-m)x^*)(y-y^*)^2\leq 0.
\end{equation*}
Since $V(x,y)$, given in \eqref{lyap_fun}, is a continuous positive definite function with negative definite time derivative in the deleted neighbourhood of the equilibrium point $(x^*,y^*)$, it satisfies the criteria for being a Lyapunov function for the equilibrium point of subsystem \eqref{subsys2}. Therefore, the equilibrium point is globally asymptotically stable.
\end{proof}
Theorem \ref{thm_local} (or Theorem \ref{thm_global}) implies that the intrinsic growth rate of the prey must exceed $\dfrac{k_1bp(1-m)^2y^*}{(1+b(1-m)x^*)^2}$ $\left(\text{or } \dfrac{k_1bp(1-m)^2y^*}{(1+b(1-m)x^*)}\right)$ for the positive equilibrium to be locally (or globally) asymptotically stable. We can also conclude that if the prey refuge can be regulated to validate the conditions in Theorems \ref{thm_local} \& \ref{thm_global}, then the co-existent state can be made locally or globally asymptotically stable. 
\subsection{Analysis of subsystem-II}\label{secsubsys2}
Here, we investigate and present a brief analysis of the subsystem of system \eqref{sys} where both species are harvested by means of a continuous linear harvesting policy. The dynamics of the system are given by the equation:
\begin{equation}\label{subsys2}
\dot{z}(t)=F_{S_2}(z),
\end{equation}
which is obtained by setting $\psi=1$ in system \eqref{sys}. Similar to subsystem \eqref{subsys1}, subsystem \eqref{subsys2} also has four equilibria, among which one is where both species coexist and its stability properties need to be studied. 

The conditions for local and global stability mentioned in Theorem \ref{thm_local} and Theorem \ref{thm_global} also hold true for the interior equilibrium point $(x^*,y^*)$ of subsystem \eqref{subsys2}.
\section{Analysis of Filippov system \eqref{sys}}\label{secfilippov}
The analysis of solutions for a Filippov (non-smooth) system involves defining the velocities of the state variables along the switching manifold where the two vector fields intersect. Filippov suggested that a solution originating from a vector field can take two possible paths once it reaches the switching manifold ($\Sigma$): either it crosses $\Sigma$ and follows the vector field of the region it enters, or it remains on $\Sigma$. 
Solutions of system \eqref{sys} can be obtained by concatenating standard solutions of the system in $S_1$ and $S_2$ along with the sliding solutions on the manifold $\Sigma$, whose governing equation is obtained by the Filippov convex method \cite{filippov1988} or Utkin's equivalent control method \cite{utkin1992}. 
For system \eqref{sys}, $S_1$ and $S_2$ are defined in \eqref{s1s2}. Both the Filippov convex method and Utkin's equivalent control method serve as useful tools to define the velocity vector on the switching manifold. According to the Filippov convex method, the velocities of the state variables on $\Sigma$ is a convex combination of the velocities in both vector fields. Meanwhile, Utkin's control method approximates the piecewise-smooth function with an alternate continuous function.
\subsection{Preliminaries}
\begin{definition}
The hyperplane $\Sigma$ is further divided into several disjoint regions. These are:
\begin{enumerate}[(I)]
\item \textbf{Crossing region:} A region on the switching line through which the trajectories may cross from one vector filed to the another. It is represented as 
$$\Sigma_C=\left\{z\in\Sigma : \sigma_1 (z)\sigma_2 (z)>0\right\},$$ 
where $\sigma_1 (z)=\langle P_z(z),F_{S_1}(z)\rangle$ and $\sigma_2 (z)=\langle P_z(z),F_{S_2}(z)\rangle$. Here, $\langle\cdot,\cdot\rangle$ denotes the standard scalar product, $P_z(z)=\nabla P(z)$ is normal to $\Sigma$ and $\nabla$ denotes the gradient, where $P(z)$ is defined in \eqref{pz}.
\item \textbf{Sliding region:} This is the interior of the complement of the crossing region on the swtching line. Dynamics originated in this region may slide along the region or may slip into either of the vector fields. It is represented as $\Sigma_S=\left\{z\in\Sigma : \sigma_1 (z)\sigma_2 (z)<0\right\}$.
\end{enumerate}
\end{definition}
\begin{definition}
The sliding region $\Sigma_S$ can be further categorized as one of the followings:
\begin{enumerate}[(i)]
\item \textbf{Escaping sliding region:} A sliding region is called escaping when the vector fields around it are directed away from the region. This causes the dynamics on the sliding region slip into one of the vector fields. The set is represented as  $\Sigma_{ES}$ and defined as $\Sigma_{ES}=\left\{z\in\Sigma_s: \sigma_1 (z) <0\; \&\; \sigma_2 (z) >0\right\}$,
\item \textbf{Attracting sliding region:} A sliding region is called attracting when the vector fields around it are directed towards the region. This causes the dynamics on and around it converge to the pseudo-equilibrium state on the switching line. The set is defined as  $\Sigma_{AS}=\left\{z\in\Sigma_s: \sigma_1 (z) >0\; \&\; \sigma_2 (z) <0\right\}$.
\end{enumerate}
\end{definition}

\noindent According to Filippov, on the sliding mode region $\Sigma_S$,
\begin{equation}
\frac{dz(t)}{dt}=F_S(z),
\end{equation}
where $F_S(z)=\lambda F_{S_1}(z)+(1-\lambda)F_{S_2}(z)$ and $$\lambda=\frac{\langle P_z(z),F_{S_2}(z)\rangle}{\langle P_z(z),F_{S_2}(z)-F_{S_1}(z)\rangle}\in (0,1).$$
\subsection{Sliding segment \& sliding mode dynamics}
A sliding segment is present when there are regions along the discontinuity boundary $\Sigma$ where the vectors for the two subsystems of system \eqref{sys} are oriented in opposite directions. It lies on the switching manifold $\Sigma$ which is given by $P(z)=0$ or $x=S$. Clearly, $P_z(z)=(1,0)^T$. So, on $\Sigma$, we have,
\begin{equation*}
\sigma(z)=\left\{r_1S\left(1-\frac{S}{k_1}\right)-\frac{p(1-m)Sy}{1+b(1-m)S}\right\}\times\left\{r_1S\left(1-\frac{S}{k_1}\right)-\frac{p(1-m)Sy}{1+b(1-m)S}-q_1ES\right\}.
\end{equation*}
By definition, we must have $\sigma(z)<0$ on the sliding region. So the sliding region can be defined as
\begin{equation*}
\Sigma_S=\left\{z\in\mathbb{R}_+^2 : x=S\; \& \; \underaccent{\bar}{y}<y<\bar{y}\right\},
\end{equation*}
where
\begin{equation*}
\bar{y}=\frac{r_1(k_1-S)(1+b(1-m)S)}{k_1p(1-m)}\quad \&\quad \underaccent{\bar}{y}=\frac{(r_1(k_1-S)-q_1Ek_1)(1+b(1-m)S)}{k_1p(1-m)}.
\end{equation*}
The governing equation of dynamics on the switching manifold $\Sigma_S$ is given by the equivalent control method introduced by Utkin \cite{utkin1992}. Following this, as $P(z)=0$ on $\Sigma_S$, so $\frac{dP}{dt}=0$ gives
\begin{equation*}
\psi=\frac{r_1\left(1-\frac{S}{k_1}\right)-\frac{p(1-m)Sy}{1+b(1-m)S}}{q_1E}.
\end{equation*}
Hence, by Utkin's method, the dynamics on $\Sigma_S$ can be given by
\begin{equation}\label{phi_y}
\frac{dy}{dt}=r_2y\left(1-\frac{y}{k_2}\right)-q_2y\frac{r_1\left(1-\frac{S}{k_1}\right)-\frac{p(1-m)y}{1+b(1-m)S}}{q_1}=\phi(y).
\end{equation}
\subsection{Equilibria and some special points}
\begin{definition}
Let $z^*$ is an equilibrium point (EP) of \eqref{sys}, then
\begin{enumerate}[(i)]
\item $z^*$ is a regular EP if $F_{S_1}(z^*)=0\;\&\; P(z^*)<0$ or $F_{S_2}(z^*)=0 \;\&\; P(z^*)>0$.
\item $z^*$ is a virtual EP if $F_{S_1}(z^*)=0\;\&\; P(z^*)>0$ or $F_{S_2}(z^*)=0 \;\&\; P(z^*)<0$.
\item $z^*$ is a pseudo-equilibrium point if it is an EP of the sliding mode region, so $\lambda F_{S_1}(z^*)+(1-\lambda) F_{S_2}(z^*)=0$, where $$\lambda = \frac{\langle P_z(z^*),F_{S_2}(z^*)\rangle}{\langle P_z(z^*),F_{S_2}(z^*)-F_{S_1}(z^*)\rangle},$$
\item $z^*$ is a boundary EP if $F_{S_1}(z^*)=0 \;\&\; P(z^*)=0$ or $F_{S_2}(z^*)=0 \;\&\; P(z^*)=0$.
\end{enumerate}
\end{definition}
\begin{definition}
A point $z\in \Sigma_S$ is called a tangent point if $\langle P_z(z^*),F_{S_1}(z^*)\rangle =0$ or $\langle P_z(z^*),F_{S_2}(z^*)\rangle=0$. 
\end{definition}
Now we investigate the existence of such points for system \eqref{sys}, which are regular equilibrium ($E_R$), virtual equilibrium ($E_V$), pseudo-equilibrium ($E_P$), boundary equilibrium ($E_B$) and tangent points($E_T$). The superscript $i$ in $E^i$ represents the vector field $F_{S_i}$ to which the equilibrium point belongs. These are discussed in the following:
\begin{enumerate}[(1).]
\item An equilibrium point $z^*=(x^*,y^*)$ of subsystem-I is called regular ($E_R^1$) if $x^*<S$ and virtual ($E_V^1$) if $x^*>S$. Similarly, an equilibrium $(x^*,y^*)$ of subsystem-II is regular ($E_R^2$) if $x^*>S$ and virtual ($E_V^2$) if $x^*<S$. Whether real or virtual, an equilibrium of each of subsystem-I and II can co-exist.
\item As a pseudo-equilibrium lies on the switching manifold, it is of the form $(S,y_1)$ and it satisfies 
\begin{equation}\label{pseudo1}
\lambda F_{S_1}(S,y_1)+(1-\lambda) F_{S_2}(S,y_1)=0,
\end{equation}
 where $$\lambda=\frac{\langle P_z(S,y_1),F_{S_2}(S,y_1)\rangle}{\langle P_z(S,y_1),F_{S_2}(S,y_1)-F_{S_1}(S,y_1)\rangle}=1-\frac{r_1\left(1-\frac{S}{k_1}\right)-\frac{p(1-m)Sy_1}{1+b(1-m)S}}{q_1E}.$$ 
Putting the value of $\lambda$ in \eqref{pseudo1}, we obtain $$y_1=\frac{q_2k_2r_1(k_1-S)(1+b(1-m)S)-q_1k_1k_2\{r_2+(r_2b+ep)(1-m)S\}}{q_2k_1k_2p(1-m)-q_1k_1r_2(1+b(1-m)S}.$$
So, $E_P=(S,y_1)$ is the pseudo-equilibrium point.
\item When an equilibrium point of either of the subsystems lie on the switching manifold, it is called a boundary equilibrium point. Hence, {\color{blue}it} lies on $x=S$ and satisfies $F_{S_1}(z)=0$ or $F_{S_2}(z)=0$. Eliminating $y$ from both equations of $F_{S_1}(S,y)=0$, we can show that $E_B^1=(S,\bar{y})$ is a boundary point, provided
\begin{equation*}
\frac{r_2}{k_2}\left(k_2- \frac{r_1(k_1-S)(1+b(1-m)S)}{k_1p(1-m)}\right)+\frac{ep(1-m)S}{1+b(1-m)S}=0.
\end{equation*}
Similarly, by eliminating $y$ from both equations of $F_{S_2}(S,y)=0$, we can show that $E_B^2=(S,\underaccent{\bar}{y})$ is a boundary point, provided
\begin{equation*}
\frac{r_2}{k_2}\left(k_2-\frac{\{r_1(k_1-S)-q_1ES\}(1+b(1-m)S)}{k_1p(1-m)}\right)+\frac{ep(1-m)S}{1+b(1-m)S}-q_2E=0.
\end{equation*}
\item Tangent points are where the vector field is perpendicular to the switching manifold. So, these points lie on $\Sigma$ and satisfy one of the conditions $\langle P_z(z^*),F_{S_1}(z^*)\rangle =0$ and $\langle P_z(z^*),F_{S_2}(z^*)\rangle=0$. Hence, it is obtained that $E_T^1=(S,\bar{y})$ and $E_T^2=(S,\underaccent{\bar}{y})$ are two tangent points.

A tangent point $E_T^1$ is visible (invisible) if the orbit $\dot{z}=F_{S_1}(z)$ starting at the tangent point belongs to $S_1$ ($S_2$) for sufficiently small $|t|\neq 0$ \cite{kuznetsov2003}. A similar definition applies for $E_T^2$ in vector field $F_{S_2}$.
\end{enumerate}
\subsection{Stability of the equilibrium points}
The stability of the equilibrium points $E^1$ of $F_{S_1}$ and $E^2$ of $F_{S_2}$ depend on the respective stability criteria derived for corresponding subsystems as well as on the threshold value $(S)$. The stability of the pseudo-equilibrium state $E_P$ depends on the governing equation of the sliding mode dynamics. This suggests that $E_P$ is stable if $\phi'(E_P)<0$ and unstable if $\phi'(E_P)>0$, where $\phi$ is defined in \eqref{phi_y}.
\begin{theorem}\label{thm_erep_main}
The regular equilibrium points $E_R^1$ and $E_R^2$ can co-exist. Furthermore, the pseudo-equilibrium point $E_P$ can exist along with $E_R^1$ and $E_R^2$. Moreover, $E_P$ is unstable when $E_R^1$ and $E_R^2$ both satisfy the condition in Theorem \ref{thm_global}.
\end{theorem}
\begin{proof}
Let $E^1=(x^*_1,y^*_1)$ and $E^2=(x^*_2,y^*_2)$. Since, $E^1$ and $E^2$ are the equilibria of $F_{S_1}$ and $F_{S_2}$ respectively, so we must have $x_1^*<S<x_2^*$ for the equilibria to be regular. Then $E^1=E_R^1$ and $E^2=E_R^2$. Thus $E_R^1$ and $E_R^2$ can co-exist. Also, the pseudo-equilibrium $E_P(S,y_1)$ exists only when $\underaccent{\bar}{y}<y_1<\bar{y}$. Thus these three equilibria can co-exist. Moreover, since $E_R^1$ and $E_R^2$ are globally stable in their respective subsystems, then solutions originating in $S_1$ will converge to $E_R^1$ and the same in $S_2$ will converge to $E_R^2$. Therefore in some neighbourhood of the sliding mode $\Sigma_S$, the direction of the vector field is outwards from $\Sigma_S$. Thus $E_P$ is unstable.
\end{proof}
\section{Numerical analysis}\label{secnum}
In this section, we present some bifurcation analysis of the Filippov system in \eqref{sys} with respect to several key parameters of the system.
\subsection{Bifurcation set of equilibria}
Based on the above analysis of the Filippov system in \eqref{sys} and its existence of sliding mode dynamics, we observed that the threshold values $S$ plays the key role in determining the complex nature of the special points. So it is necessary to investigate the impact of $S$ on the behaviour of the equilibria and sliding dynamics. Therefore, we present the bifurcation set of equilibria for two key parameters of the system which are the threshold value $(S)$ and the predation rate $(p)$. Two instances are presented here with different types of dynamical outputs driven by two different sets of parameters.

$A_1 = \{(r_1,k_1,m,p,b,q_1,E,r_2,k_2,e,q_2)$ : $r_1 = 0.9$, $k_1=2$, $m=0.2$, $p=0.6$, $b=0.4$, $q_1=0.2$, $E=1$, $r_2=0.8$, $k_2=1.5$, $e=0.6$ \& $q_2=0.1\}$

$A_2 = \{(r_1,k_1,m,p,b,q_1,E,r_2,k_2,e,q_2)$ : $r_1=2.3$, $k_1=9$, $m=0.15$, $p=0.2$, $b=0.04$, $q_1=0.1$, $E=1$, $r_2=1.2$, $k_2=7$, $e=0.5$ \& $q_2=0.2$

Equation \eqref{xroot} is used to obtain the four curves that are necessary to analyze the equilibria. $L_1^i$'s ($i=1,2,3,4$) are the four curves for the first set of parameters and $L_2^i$'s ($i=1,2,3,4$) are the corresponding curves for the second set of parameters. $L_1^1$ and $L_2^1$ represent the maximum possible $p$ values, for the corresponding parameter sets, for which the equilibrium point $E^1$ exists, whereas $L_1^2$ and $L_2^2$ denote that of $E^2$. Also, $L_1^3$ and $L_2^3$ represent the unique positive root of equation \eqref{xroot} as a function of $p$ (denoted as $f_1^1(p)$ \& $f_2^1(p)$), for the corresponding parameter sets, for $\psi=0$, whereas $L_1^4$ and $L_2^4$ stand for the unique positive root of equation \eqref{xroot} for $\psi=1$.\\
\begin{minipage}{0.5\textwidth}
\begin{eqnarray*}
L_1^1 &=& \{(S,p): p=0.75\},\\
L_1^2 &=& \{(S,p): p=0.6667\},\\
L_1^3 &=& \{(S,p): S=f_1^1(p)\},\\
L_1^4 &=& \{(S,p): S=f_1^2(p)\},
\end{eqnarray*}
\end{minipage}
\begin{minipage}{0.5\textwidth}
\begin{eqnarray*}
L_2^1 &=& \{(S,p): p=0.38655\},\\
L_2^2 &=& \{(S,p): p=0.4437\},\\
L_2^3 &=& \{(S,p): S=f_2^1(p)\},\\
L_2^4 &=& \{(S,p): S=f_2^2(p)\},
\end{eqnarray*}
\end{minipage}

\begin{figure}[H]
 \subfloat[\label{sp_bif}]{%
  \includegraphics[width=0.5\textwidth, height=6cm]{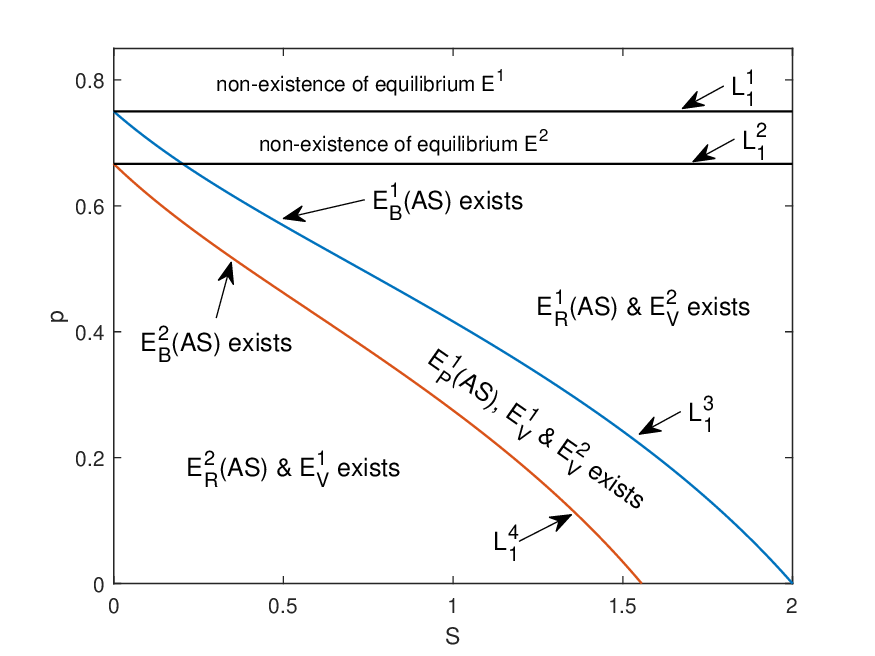}}
 \subfloat[\label{sp_bif1}]{%
  \includegraphics[width=0.5\textwidth, height=6cm]{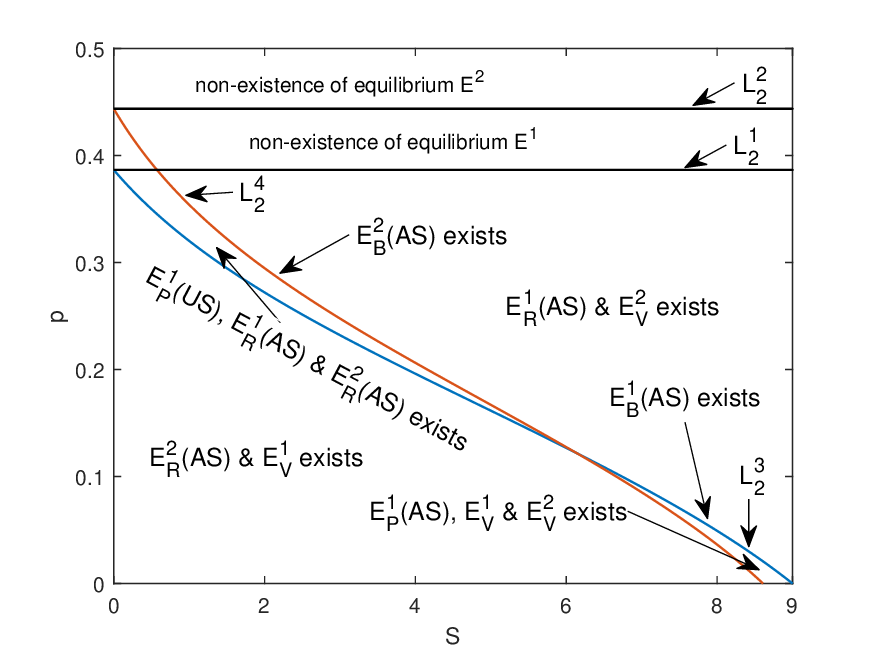}}
\caption{Bifurcation diagram of system \eqref{sys} with respect to the predation rate $p$ and the threshold $S$ showing existence of different type of equilibria with their stability properties like asymptotically stable (AS) or unstable (US): {\bf (a)} For parameters in set A$_1$ with varying $S$ \& $p$ {\bf (b)} For parameters in set A$_2$ with varying $S$ \& $p$.}
\label{S_vs_p_bif}
\end{figure}

Here, an interesting case arises when $S$ and $p$ values lie inside the region bounded below by $L_2^3$ and bounded above by $L_2^1$ and $L_2^4$ in Fig. \ref{sp_bif1}. In such cases both harvesting and non-harvesting positive equilibrium become regular and are stable. Also, these regular equilibria co-exist with a pseudo-equilibrium point which becomes unstable. Figure \ref{basins} shows the sliding mode along with the basins of attraction of the two regular equilibria. $E_T^1$ and $E_T^2$ both become visible tangent points. The pseudo-equilibrium lies on the separatrix that divides both basins of attractions.
\begin{figure}[H]
 \subfloat[Time series\label{ts_eq1}]{%
  \includegraphics[width=0.3\textwidth, height=4.5cm]{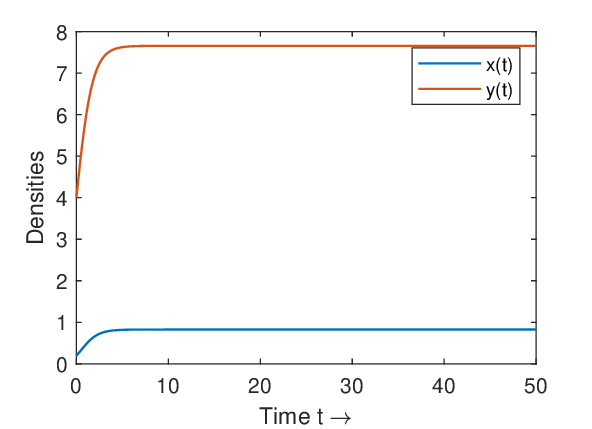}}
 \subfloat[Basins of attraction\label{basins}]{%
  \includegraphics[width=0.4\textwidth, height=4.5cm]{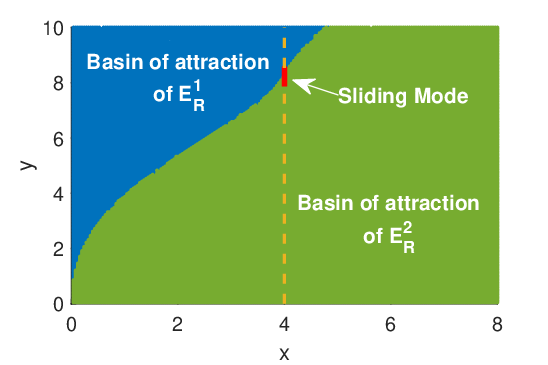}}
 \subfloat[Time series\label{ts_eq2}]{%
  \includegraphics[width=0.3\textwidth, height=4.5cm]{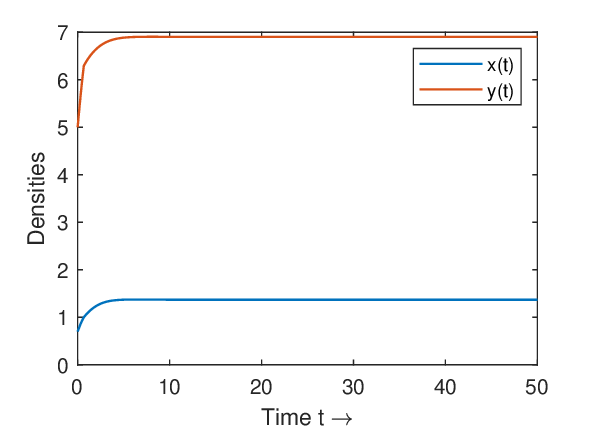}}
  \caption{Dynamics of \eqref{sys} for parameters in set A$_2$ with $S=4$ {\bf (a)} Time series starting from (0.2,4) reaches $E_R^1$ {\bf (b)} Basins of attraction of the two regular equilibria {\bf (c)} Time series starting from (0.8,5) reaches $E_R^2$.}
\label{basins}
\end{figure}
From Fig. \ref{basins}, it can be observed  that system \eqref{sys} can show multi-stability, i.e., the system possesses two stable equilibrium states simultaneously. In this case, the dynamics of certain population will depend on the position of the initial population densities relative to the basins of attraction of the equilibria $E_R^1$ and $E_R^2$.
\subsection{Impact of prey refuge}
In this subsection, we performed numerical analysis to observe the impact of the prey refuge parameter $m$ on the behaviour of the system in \eqref{sys}. We need to vary the value of $m$ and record the changes in the system dynamics. We have obtained bifurcation diagrams for varying $S$ and $p$ with different $m$-values which consists three important key aspects of the system (see Fig. \ref{bif_m}).
\begin{figure}[H]
 \subfloat[$m=0.4$\label{bif_m:1}]{%
  \includegraphics[width=0.5\textwidth, height=6cm]{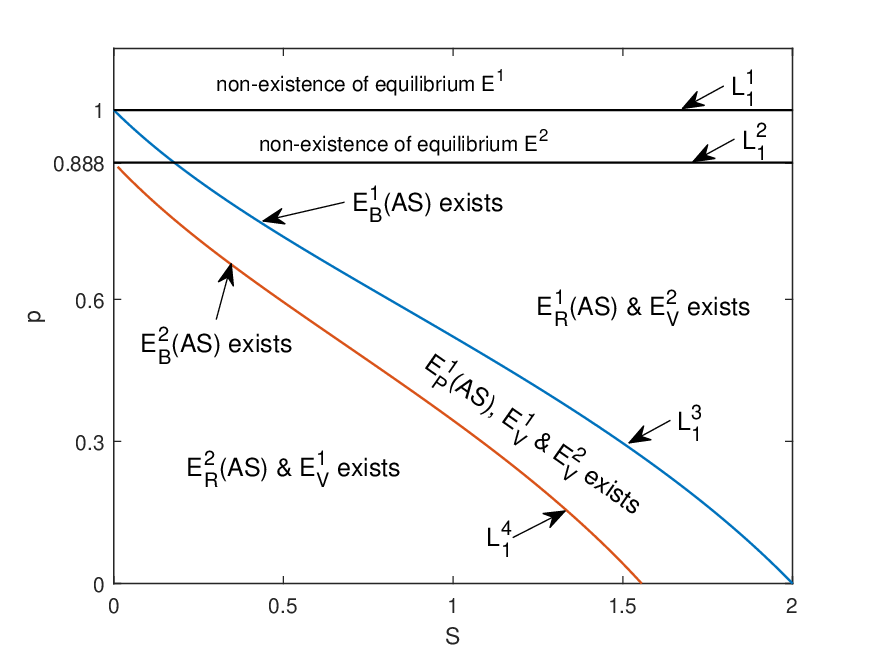}}
 \subfloat[$m=0.8$\label{bif_m:2}]{%
  \includegraphics[width=0.5\textwidth, height=6cm]{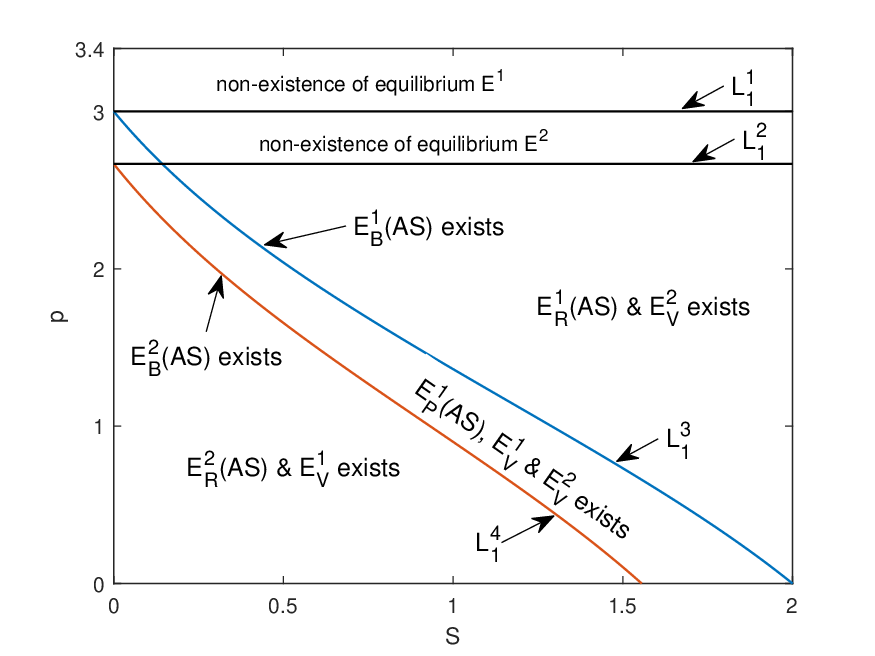}}\\
 \subfloat[$m=0.9$\label{bif_m:3}]{%
  \includegraphics[width=0.5\textwidth, height=6cm]{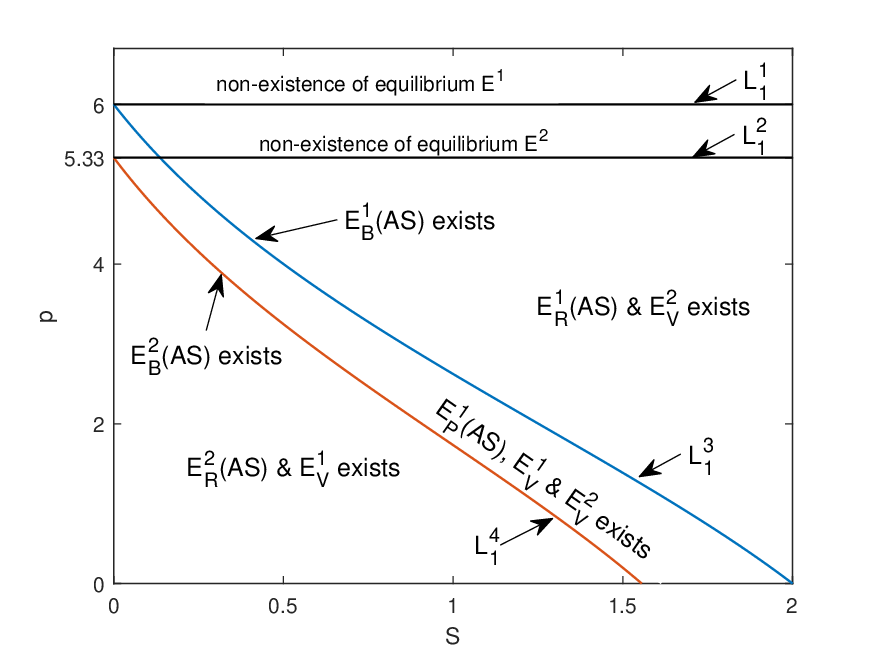}}
\caption{Bifurcation diagrams in $Sp$-plane for different $m$-values with other parameters from set A$_1$}
\label{bif_m}
\end{figure}

From Fig. \ref{bif_m}, it is apparent that increasing the value of $m$ results in the existence of the positive equilibrium points $E^1_R$ and $E^2_R$ for larger values of $p$. This observation suggests that higher values of $m$ correspond to expanded basins of attraction for the interior equilibrium points in $S-p$ parameter space.
\subsection{Sliding Bifurcation}
Here, we investigate the existence of sliding bifurcation for system \eqref{sys} by varying the threshold value for harvesting while all other parameters constant.

{\em Boundary focus bifurcation:} 
This is a type of local sliding bifurcation which may occur when a regular equilibrium point and a pseudo-equilibrium point collide together at a tangent point once the parameter $S$ crosses a threshold value \cite{tan2016}. Let us fix the values of the system parameters, except $S$, as follows: $r_1=0.9$, $k_1=2$, $p=0.6$, $m=0.2$, $b=0.4$, $q_1=0.2$, $E=1$, $r_2=0.8$, $k_2=1.5$, $e=0.6$ \& $q_2=0.1$. Now, numerically it can be obtained that $E^1(0.4005,1.6917))$ is an equilibrium point of vector field $F_{S_1}$ and $E^2(0.1416,1.3856)$ is that of $F_{S_2}$. These equilibrium points can be further termed as regular $(E_R)$, virtual $(E_V)$ or boundary $(E_B)$ as per their definitions.

In Fig. \ref{bif_focus} and \ref{bif_node}, the the vertical dotted line represents the switching manifold $P(z)=0$ or $x=S$, and the red line segment represents the sliding region. The green dot denotes that the equilibrium is locally asymptotically stable, the black dot denotes that the equilibrium is irrelevant for the dynamics of the system whereas the tangent points are shown by red dots.

For $S<0.1416$, there exists one regular equilibrium point $E_R^2$ and one virtual equilibrium point $E_V^1$ (see Fig. \ref{s0p1}). The pseudo-equilibrium point does not exists since $E_P$ lies outside the sliding region $\underaccent{\bar}{y}<y<\bar{y}$. $E_T^1$ becomes an invisible tangent point while $E_T^2$ is a visible tangent point. This suggests that when critical prey population density is small enough, system \eqref{sys} can attain the harvesting equilibrium in the long-run. This takes care of the harvesting and sustainability of the populations simultaneously. 

For $S=0.1416$, the regular equilibrium $E_R^2$, the pseudo-equilibrium $E_P$ and the tangent point $E_T^2$ collide and form the boundary equilibrium point $E_B^2$ which becomes a stable equilibrium state. Hence, $S=0.1416$ becomes the bifurcation point for boundary focus bifurcation. The scenario is depicted in Fig. \ref{s0p141}. 

For $0.1416<S<0.4005$, the two equilibrium points of the two vector fields become virtual equilibrium states $E_V^1$ \& $E_V^2$ (see Fig. \ref{s0p25}). Both tangent points become invisible and the sliding mode region $\underaccent{\bar}{y} <y<\bar{y}$ contains the stable pseudo-equilibrium $E_P$. Since the dynamics on the switching manifold is a convex combination of the harvesting and non-harvesting subsystem, so some level of harvesting still continues when the dynamics reaches $E_P$.
\begin{figure}[H]
 \subfloat[For $S=0.1$, $E_R^2$ becomes asymptotically stable\label{s0p1}]{%
  \includegraphics[width=0.5\textwidth, height=6cm]{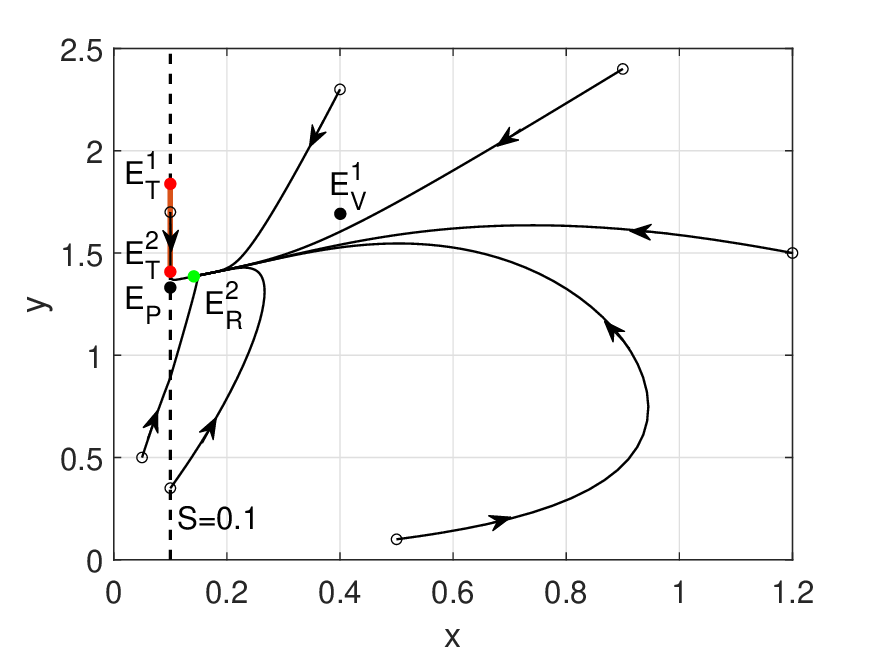}}
 \subfloat[For $S=0.1416$, a stable equilibrium $E_B^2$ emerges\label{s0p141}]{%
  \includegraphics[width=0.5\textwidth, height=6cm]{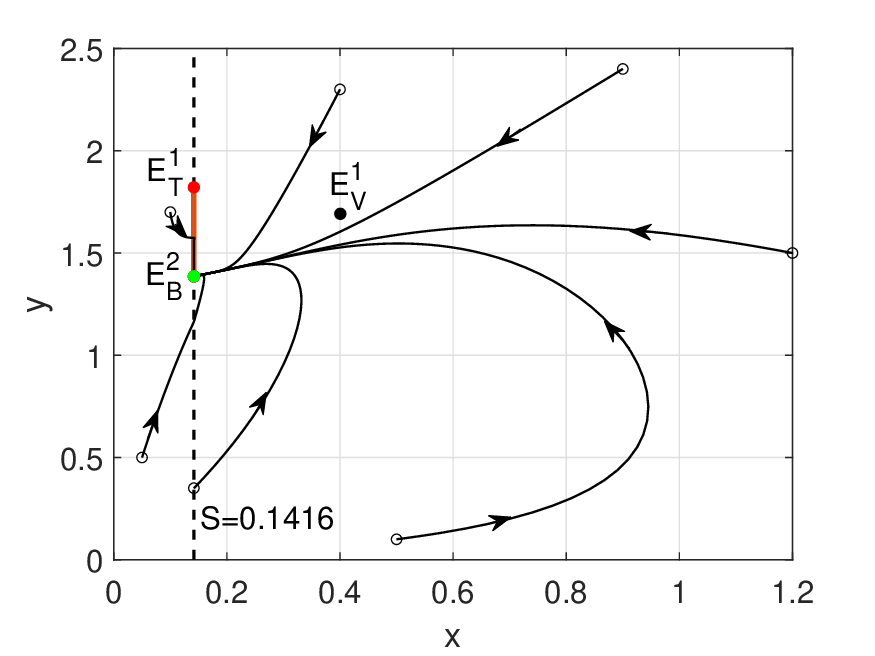}}\\
 \subfloat[For $S=0.25$, the pseudo-equilibrium $E_P$ becomes asymptotically stable\label{s0p25}]{%
  \includegraphics[width=0.5\textwidth, height=6cm]{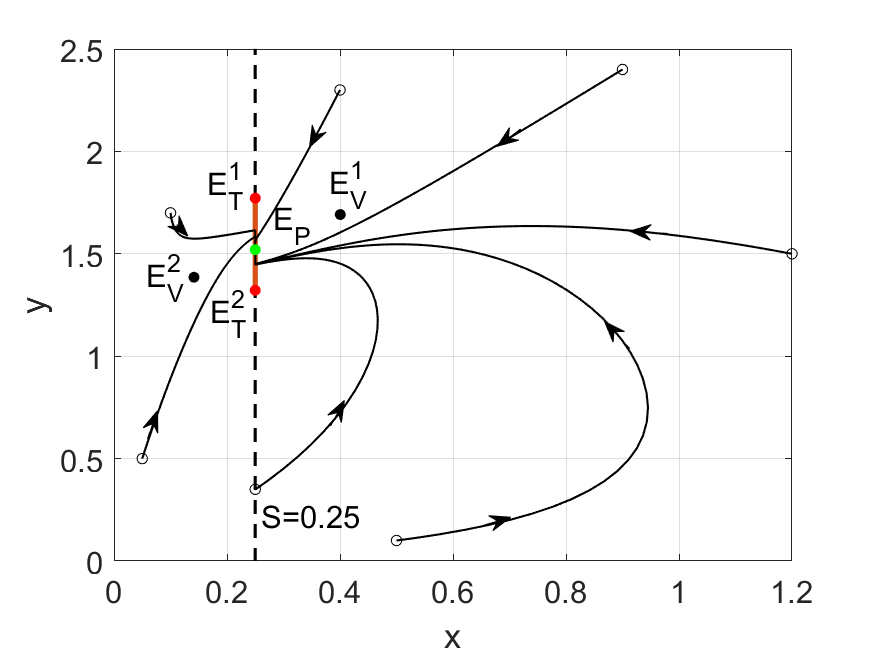}}
\caption{System \eqref{sys} undergoes boundary focus bifurcation at $S=0.1416$ when $r_1=0.9$, $k_1=2$, $p=0.6$, $m=0.2$, $b=0.4$, $q_1=0.2$, $E=1$, $r_2=0.8$, $k_2=1.5$, $e=0.6$ \& $q_2=0.1$}
\label{bif_focus}
\end{figure}

{\em Boundary node bifurcation:} 
For $S=0.4005$, the virtual equilibrium point $E_V^1$, the tangent point $E_1^2$ and the pseudo-equilibrium point $E_P$ collide to form the boundary equilibrium point $E_B^1$ (see Fig. \ref{s0p4005}). Thus boundary node bifurcation occurs. Therefore, $S=0.4005$ becomes the bifurcation point. {\color{blue}}

For $S>0.4005$, the regular equilibrium point of vector field $F_{S_1}$ becomes stable whereas the only equilibrium point of $F_{S_2}$ becomes virtual (see Fig. \ref{s0p6}). The pseudo-equilibrium point $E_P$ does not exist and the tangent point $E_T^1$ becomes visible while $E_T^2$ is invisible. This implies that system \eqref{sys} will reach the non-harvesting equilibrium eventually. So higher critical prey threshold value leads to a situation in the long-term where harvesting is not implemented at all.
\begin{figure}[H]
 \subfloat[For $S=0.4005$, a stable equilibrium $E_B^1$ emerges \label{s0p4005}]{%
  \includegraphics[width=0.5\textwidth, height=6cm]{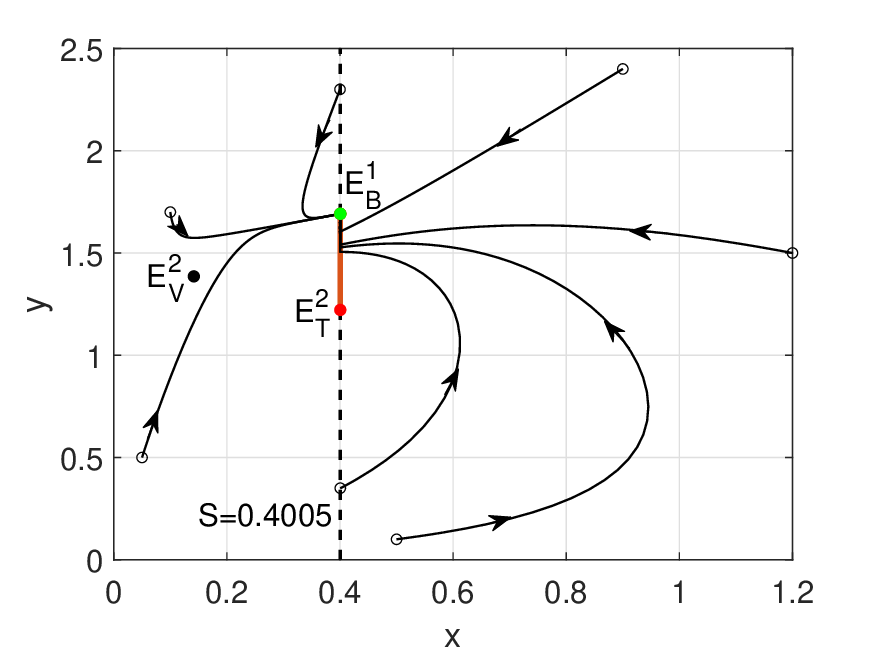}}
 \subfloat[For $S=0.7$, $E_R^1$ becomes asymptotically stable\label{s0p6}]{%
  \includegraphics[width=0.5\textwidth, height=6cm]{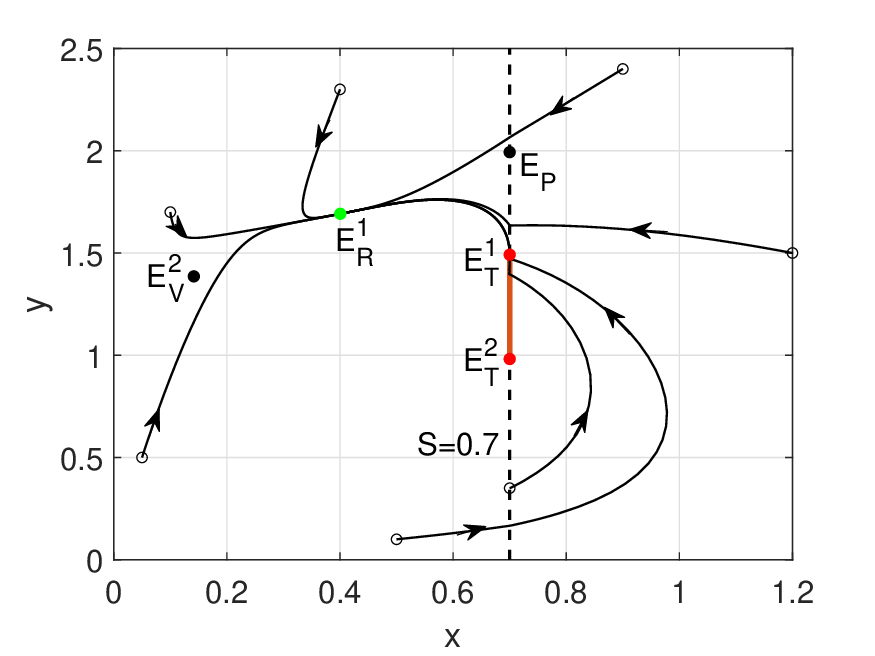}}
\caption{System \eqref{sys} undergoes boundary node bifurcation at $S=0.4005$ when $r_1=0.9$, $k_1=2$, $p=0.6$, $m=0.2$, $b=0.4$, $q_1=0.2$, $E=1$, $r_2=0.8$, $k_2=1.5$, $e=0.6$ \& $q_2=0.1$}
\label{bif_node}
\end{figure}
\section{Conclusion and Discussion}\label{seccon}
Here, a Filippov (non-smooth) system has been used to study the predator-prey interactions. 
The model maybe relevant for examining the dynamics of two fish populations, wherein the prey species represents a herbivorous fish population that has access to refuge (e.g.: some reef-dwelling fish species, Peruvian anchovy etc.) and the predator species is a generalist (e.g.: Lionfish, common dolphinfish, etc.) that preys on various small fishes. 
The Filippov system effectively regulates the harvesting using a discontinuous harvesting strategy which depends on the density of the prey population. This helps maintain the density at a certain threshold level. As pointed out in \cite{filippov1988}, the method of study for these systems are different from usual smooth systems. The domain and dynamics of the sliding region are obtained using Utkins's equivalent control method \cite{utkin1992}. The sliding mode dynamics and several types of equilibria, for e.g. regular, virtual and pseudo equilibria can provide new insights about the systems.

The threshold parameter plays the key role in deciding the equilibrium points which can contribute to the population dynamics. While the pseudo-equilibrium present in the sliding region and the regular equilibria can affect the local and global dynamics, the pseudo-equilibrium present outside the sliding region and the virtual equilibria are like any other ordinary points in the phase portrait. As the threshold level varies, the population can either settle at one of the regular equilibrium points or stabilize at a pseudo-equilibrium.

It is necessary to study the subsystems in \eqref{subsys1} and \eqref{subsys2} to study the Filippov system. Therefore, the existence of axial and interior equilibria is demonstrated, along with the derivation of uniqueness criteria for their respective interior equilibrium. Furthermore, the local and global stability criteria for their corresponding equilibrium are obtained. The solution of the Filippov system are shown to be bounded in the positive quadrant. For the numerical analysis of system \eqref{sys}, we have considered two different sets of parameter values  which are set A$_1$ and A$_2$. The bifurcation set of equilibria are obtained for these two sets of parameters and they show very different type of dynamical behaviour when the $S-p$ parameter space is explored (see Fig. \ref{S_vs_p_bif}). We observed that, the pseudo-equilibrium is stable when both positive equilibria are virtual. As stated in Theorem \ref{thm_erep_main}, Fig. \ref{sp_bif1} clearly shows that when both regular equilibria can co-exist and become stable, the pseudo-equilibrium becomes unstable. Therefore, whether the populations will attain the non-harvesting equilibrium $(E_R^1)$ or the harvesting equilibrium $(E_R^2)$ in the long-run depends on the initial conditions. Thus certain initial densities can be chosen to reach the harvesting equilibrium so that harvesting can be implemented while maintaining certain threshold density.

For the parameters in set A$_1$, we observed that the system undergoes some sliding bifurcations like boundary focus and boundary node bifurcation. We observed that, for lower threshold parameter values $(S\leq 0.1416)$ the harvesting equilibrium is stable and for higher threshold values $(S\geq 0.4005)$, the non-harvesting equilibrium becomes stable. The system dynamics stabilizes at the pseudo-equilibrium for medium threshold values $(0.1416<S<0.4005)$ that is some level of harvesting will continue to occur since the equilibrium is a convex combination of both harvesting and non-harvesting equilibria. So we can conclude that by selecting a suitable threshold value (critical prey population density), we can bring the populations to the harvesting equilibrium, ensuring that harvesting can continue while maintaining critical population densities. Also, observing the bifurcation diagrams for different $m$-values (see Fig. \ref{bif_m}), it is noticed that the basin of attraction for the regular equilibrium points in the $S-p$ parameter space increases as the $m$-values increase.
%

\section*{Funding}
Not applicable.
\section*{Conflict of interest/Competing interests}
The authors declare that they have no conflict of interest.
\section*{Availability of data and material (data transparency)}
Not applicable
\section*{Code availability}
Not applicable.

%
%

\end{document}